\documentclass[12pt]{amsart}
\usepackage{amssymb,amsmath,amscd,graphicx,
latexsym,amsthm}
\usepackage{amssymb,latexsym,eufrak,amsmath,amscd,graphicx}
  \usepackage[all]{xy}
\theoremstyle{plain}
\newtheorem{theorem}{Theorem}[section]

\newtheorem{corollary}[theorem]{Corollary}
\newtheorem{lemma}[theorem]{Lemma}

\newtheorem{theoremalpha}{Theorem}
\newtheorem{corollaryalpha}[theoremalpha]{Corollary}

\theoremstyle{definition}

\newtheorem{remark}[theorem]{Remark}
\newtheorem{example}[theorem]{Example}
\newtheorem{question}[theorem]{Question}

\newcommand{\lra}{\longrightarrow}
\newcommand{\noi}{\noindent}

\newcommand{\NN}{\mathbf{N}}

\newcommand{\CC}{\mathbf{C}}
\newcommand{\QQ}{\mathbf{Q}}

\newcommand{\OO}{\mathcal{O}}

\newcommand{\JJ}{\mathcal{J}}

\newcommand{\frb}{\mathfrak{b}}

\newcommand{\bull}{_{\bullet}}

\newcommand{\frakm}{\mathfrak{m}}

\newcommand{\rndup}[1]{ \ulcorner {#1}
\urcorner }

\newcommand{\MI}[1]{\mathcal{J} \big ( {#1}
\big) }

\newcommand{\MMI}[2]{\MI{ {#1} \, , \, {#2}}}
\newcommand{\MMIP}[2]{\MI{ {#1} \, ; \, {#2}}}

\newcommand{\pr}{\prime}

\newcommand{\num}{ \equiv_{\text{num}} }

\newcommand{\MIJ}  {\mathcal{J}}
\newcommand{\Tor}{\textnormal{Tor}}
\newcommand{\Image}{\textnormal{Im}}
\newcommand{\LL}{\mathcal{L}}
\newcommand{\frc}{\mathfrak{c}}
\newcommand{\tn}[1]{\textnormal{#1}}

\begin{document}

\title{Syzygies of multiplier ideals on singular varieties}

\author{Robert Lazarsfeld}
\address{Department of Mathematics, University of Michigan, Ann Arbor, MI
 48109}
\email{{\tt rlaz@umich.edu}}
\thanks{Research of the first author partially supported by NSF grant DMS-0652845,}

\author{Kyungyong Lee}
\address{Department of Mathematics, University of Michigan, Ann Arbor, MI
 48109}
\email{{\tt kyungl@umich.edu}}

\author{Karen E. Smith}
\address{Department of Mathematics, University of Michigan, Ann Arbor, MI
 48109}
 \thanks{Research of the third author partially supported by NSF grant DMS-0500823.}

\email{{\tt kesmith@umich.edu}}

\dedicatory{Dedicated to Mel Hochster on the occasion of his sixty-fifth birthday}
\maketitle

\section*{Introduction}
It was recently established by the first two authors  in  \cite{LL}
that multiplier ideals  on a smooth variety satisfy some special
syzygetic properties. The purpose of this note is  to show how some
of these can be extended to the singular setting.

To set the stage, we review some of the results  from \cite{LL}. Let
$X$ be a smooth complex variety of dimension $\dim(X) = d$, and
denote by $(\OO, \frakm)$ the local ring of $X$ at a fixed point $x
\in X$.  Let   $\MIJ \subseteq \OO$ be any multiplier ideal:  i.e.
assume that   $\MIJ$ is the stalk at $x$ of a multiplier ideal sheaf
$\MI{X, \frb ^\lambda}$, where  $\frb \subseteq \OO_X$  is  an ideal
sheaf   and $\lambda$ is a positive rational number. The main result
of \cite{LL} is that if $p \ge 1$ then no minimal $p$-th syzygy of
$\MIJ$ vanishes  modulo $\frakm^{d + 1- p}$ at $x$.  In other words,
if we consider a minimal free resolution of the ideal $\MIJ$ over
the regular local ring $\OO$
\[
\xymatrix{
 \ldots \ar[r]^{u_3} &F_2 \ar[r]^{u_2}
 &   F_1 \ar[r]^{u_1} & F_0 \ar[r]   & \MIJ \ar[r] & 0
  },
\] then no minimal
generator of the $p^{\text{th}}$ syzygy module
\[ \textnormal{Syz}_p(\mathcal{J}) \ =_{\text{def}} \ \Image(u_p) \
\subseteq \ F_{p-1}\] of $\mathcal{J}$ lies in $\frakm^{d+ 1 -
p}\cdot F_{p-1}$. While this result places no restriction on the
orders of vanishing of the {\it generators\/} of $\MIJ$, it provides
strong constraints on the first and higher {\it syzygies}  of
$\MIJ$. When $d = 2$ these conditions hold for any integrally closed
ideal, but they show that in dimensions $d \ge 3$ only rather
special integrally closed ideals can arise as multiplier
ideals.\footnote{By contrast, it was established by Favre-Jonsson
\cite{FJ} and Lipman-Watanabe \cite{LW} that any integrally closed
ideal on a smooth \textit{surface} is locally a multiplier ideal.}

 Multiplier ideals can be defined on any $\QQ$-Gorenstein variety $X$, or more generally for any pair $(X, \Delta)$ consisting of an effective Weil $\QQ$-divisor $\Delta$ on a normal variety $X$ such that $K_X + \Delta$ is $\QQ$-Cartier. It is natural to wonder whether multiplier ideals in this context satisfy the same sort of algebraic properties as in the smooth case.   We will see (Example \ref{example}) that the result from \cite{LL} just quoted does not extend  without change. On the other hand,  we show that at least for \textit{first} syzygies, one gets a statement by replacing the maximal ideal $\frakm$ by any parameter ideal.
 \begin{theoremalpha} \label{TheoremA} Let $(X, \Delta)$ be a pair with $\dim X = d$,   let $(\OO, \frakm)$ be the local ring of $X$ at a Cohen-Macaulay point $x \in X$, and fix a system of parameters
 \[   z_1, \ldots , z_d \ \in \ \OO. \]  Let $\MIJ  \subseteq
 \OO$ be \tn{(}the germ at $x$ of\tn{)} any multiplier ideal on $(X, \Delta)$. Then  no minimal first  syzygy   of $\MIJ$ vanishes modulo $ (z_1, \dots, z_d)^{d} $.
 \end{theoremalpha}
\noi If $X$ is $\QQ$-Gorenstein then one can take $\Delta = 0$, so
that one is dealing with usual multiplier ideals of the form $\MI{X,
\frb^\lambda}$. Of course, the strongest statement is achieved by
taking $z_1, \dots, z_d$ to  generate the largest possible ideal,
which is to say,  by taking the $z_i$  to  generate a  reduction of
$\frakm$. In this case, if $x$ is a smooth point, then the $z_i$
generate the maximal ideal itself, and we recover the original
result from \cite{LL} in the case $p =1$.

Observe that while Theorem A   doesn't  give a uniform bound  on the
order of vanishing of syzygies of a multiplier ideal, it does
uniformly bound the highest  power of any ideal generated by a {\it
system of parameters \/} that can contain a syzygy. It also  yields
uniform statements provided that one brings the multiplier ideal of
the trivial line bundle into the picture. For example:
\begin{corollaryalpha}
Let $x \in X$ be a Cohen-Macaulay point with maximal ideal $\frakm$,
and set \[ \tau \ = \  \MMIP{(X, \Delta)}{\OO_X}_x. \] If $\JJ$ is
the germ at $x$ of any multiplier ideal, then no first syzygy of
$\JJ$ vanishes modulo $\tau \cdot \frakm^{2d-1}$.
\end{corollaryalpha} \noi In particular, if $(X,\Delta)$ is Kawamata
log-terminal, then no first syzygy can vanish modulo
$\frakm^{2d-1}$.

Unlike the results for  smooth varieties in \cite{LL}, our
statements here deal only with  {\it first}  syzygies. This may be
more an artifact of our method rather than a necessary restriction:
it would be interesting to investigate this further.

Concerning the organization of the paper, we start in \S 1 with a
discussion of  Skoda's theorem in the singular setting. In \S2 we
modify the arguments from \cite{LL} to prove Theorem A. We conclude
in \S 3 with some examples and applications.

We are grateful to Craig Huneke for some valuable discussions.

\section{Skoda Complexes on Singular Varieties}

In this section, we discuss the circle of ideas surrounding Skoda's
theorem in the singular setting. This appears only briefly in
\cite{PAG}, and so we thought it would be useful to spell out some
of the details. We don't claim any essential novelty for the
material in this section.

Let $(X, \Delta)$ be a pair in the sense of \cite[9.3.55]{PAG}: this
means that $X$ is a normal variety, and $\Delta = \sum d_i D_i$ is
an effective Weil $\QQ$-divisor such that $K_X + \Delta$ is
$\QQ$-Cartier.  Fix ideals $ \frb,  \frc \subseteq \OO_X$, and let
\[ \mu : X^\pr \lra X\] be a log resolution of $(X, \Delta)$, $\frb$
and $\frc$. Then one can attach numbers
\[ a(E) \, \in \, \QQ \  \ \ , \ \ b(E) \, , \,  c(E) \, \in \, \NN \]
to each exceptional divisor  of $\mu$, as well as the proper
transforms of the divisors appearing in the support of $\Delta$ or
the zeroes of $\frb$ and $ \frc$,    characterized by the
expressions
\begin{gather*}  K_{X^\pr} \num  \mu^*(K_X + \Delta) + \sum a(E)E \\
\frb \cdot \OO_{X^\pr} \, = \, \OO_{X^\pr}\big( -\sum b(E)E\, \big)
\\ \frc \cdot \OO_{X^\pr} \, = \, \OO_{X^\pr}\big( -\sum c(E)E\,
\big).
\end{gather*}
Given a rational or real weighting coefficient $\lambda > 0$, one
then defines the multiplier ideal
\[   \MMIP{(X,\Delta)}{\frc \cdot \frb^\lambda} \ = \ \mu_*  \OO_{X^\pr} \big(  \sum (\rndup{ a(E) -c(E) - \lambda b(E)} )E \big), \]
this being independent of the resolution. When $X$ is
$\QQ$-Gorenstein one can take $\Delta = 0$, and when in addition $X$
is actually Gorenstein, one has the more familiar definition
\[  \MMI{X}{\frc \cdot \frb^\lambda} \ =  \ \mu_* \OO_{X^\pr} \big( K_{X^\pr/X} - [ C + \lambda B] \big), \]
where we write
\[  B = \sum b(E) E \ \ , \ \ C = \sum c(E) E. \]

The following lemma expresses an elementary but important property
of multiplier ideals. \begin{lemma} \label{module.property} For any
integer $m \ge 0$, there is an inclusion
\[  \frc \cdot \MMIP{(X, \Delta)}{\frc^m \cdot \frb^\lambda} \  \subseteq \ \MMIP{(X, \Delta)}{\frc^{m+1} \cdot \frb^\lambda}. \]
\end{lemma}
\begin{proof}
One has
\begin{align*}
 \frc \cdot \MMIP{(X, \Delta)}{\frc^m \cdot \frb^\lambda} \ &\subseteq \ \mu_* \OO_{X^\pr}\big(-\sum c(E)E \big) \cdot  \mu_*  \OO_{X^\pr} \big(  \sum (\rndup{ a(E) -mc(E) - \lambda b(E)} )E \big)\\
 &\subseteq \ \mu_*  \OO_{X^\pr} \big(  \sum (\rndup{ a(E) -(m+1)c(E) -\lambda b(E)} )E \big) \\
 &= \ \
 \MMIP{(X, \Delta)}{\frc^{m+1} \cdot \frb^\lambda}.
 \end{align*}
\end{proof}

\begin{corollary}
With $(X, \Delta)$ as above, one has \[  \frc \cdot
\MMIP{(X,\Delta}{\OO_X} \ \subseteq \ \MMIP{(X,\Delta)}{\frc} \] for
any ideal $\frc$. In particular, if $(X, \Delta)$ is KLT, then
\[ \frc \ \subseteq \ \MMIP{(X, \Delta)}{\frc}. \qed \]
\end{corollary}

We now turn to Skoda complexes. Our interest being local, we will
assume for simplicity of notation that  $X$ is affine. Chose
elements
\[  f_1, \ldots, f_r  \,  \in \, \frc, \]
and for compactness write \[ \MI{ \frc^m \cdot \frb^\lambda}  =
\MMIP{(X, \Delta)}{\frc^m \cdot \frb^\lambda}. \] It follows from
the previous corollary that each  $f_i $ multiplies $\MI{\frc^\ell
\cdot \frb^\lambda} $ into $\MI{\frc^{\ell+1} \cdot \frb^\lambda} $.
Therefore the $f_i$ determine a complex $\textnormal{Skod}\bull(m;
f)$
\begin{equation*}
\xymatrix{ \ldots \ar[r]& \OO_X^{\binom{r}{2}} \otimes \MI{
\frc^{m-2} \cdot \frb^\lambda} \ar[r] & \OO_X^r \otimes \MI{
\frc^{m-1} \cdot \frb^\lambda} \ar[r] & \MI{ \frc^m \cdot
\frb^\lambda} \ar[r] & 0}
\end{equation*}
 arising as a subcomplex of the Koszul complex $K\bull(f_1, \ldots, f_r) = \textnormal{Kosz}\bull(f)$ on the $f_i$.

The basic result for our purposes is
\begin{theorem} \label{Skoda.Cxs.Exact}
Assume that $m \ge r$ and that the $f_i$ generate a reduction of
$\frc$. Then $\textnormal{Skod}\bull(m; f)$ is exact.
\end{theorem}
\noi Recall that the hypothesis on the $f_i$ is equivalent to asking
that their pull-backs to the log resolution $X^\pr$ generate the
pullback $\OO_{X^\pr}(-C)$ of $\frc$.

\begin{proof} [Sketch of proof] The pull-backs of the given elements $f_i \in \frc$ determine an exact Koszul complex of vector bundles on $X^\pr$:
\begin{equation*}
\xymatrix{ \ldots \ar[r]& \OO_{X^\pr}^{\binom{r}{2}} \otimes
\OO_{X^\pr} (2C) \ar[r] & \OO_{X^\pr}^r \otimes \OO_{X^\pr} (C)
\ar[r] & \OO_{X^\pr} \ar[r] & 0} .
\end{equation*}
Twisting through by $ \OO_{X^\pr} \big(  \sum (\rndup{ a(E) -mc(E) -
\lambda b(E)} )E \big)$, we get an exact sequence all of whose terms
have vanishing higher direct images thanks to the local vanishing
theorems for multiplier ideals \cite[9.4.17]{PAG}. The direct image
of this twisted Koszul complex -- which is the Skoda complex
$\textnormal{Skod}\bull(m; f)$ -- is therefore exact.
\end{proof}

We conclude this section by recording some consequences of Brian\c
con-Skoda type.
\begin{corollary} \label{Cor.BS.Type}
Assume as in the theorem that $m \ge r$ and that $f_1, \ldots, f_r$
generate a reduction of $\frc$. Then
\begin{gather}
\MMIP{(X, \Delta)}{\frc^m \cdot \frb^\lambda} \ = \ (f_1, \ldots, f_r) \cdot \MMIP{(X, \Delta)}{\frc^{m-1} \cdot \frb^\lambda} \tag{i}; \\
\frc^m \cdot \MMIP{(X,\Delta)}{\OO_X} \ \subseteq \ (f_1, \ldots,
f_r)^{m+1 -r} \tag{ii}
\end{gather}
In particular, if $(X, \Delta)$ has only log terminal singularities,
then
\[ \frc^m \ \subseteq \ (f_1, \ldots, f_r)^{m+1 -r}.\]
\end{corollary}

\begin{proof}
The first statement follows from the surjectivity of the last map in
the Skoda complex, and it implies inductively that $ \JJ_m \ = \
(f_1, \ldots, f_r)^{m+1-r}\JJ_{r-1}$. This being said, for (ii) one
uses Lemma \ref{module.property} to conclude
\begin{align*}
\frc^m \cdot \MMIP{(X, \Delta)}{\OO_X} \ &\subseteq \ \MMIP{(X, \Delta)}{\frc^m} \\
&= \ (f_1, \ldots, f_r)^{m+1-r} \cdot \MMIP{(X, \Delta)}{\frc^{r-1}}\\
&\subseteq \ (f_1, \ldots, f_r)^{m+1-r}.
\end{align*}
The last statement follows from (ii) since  $\MMIP{(X,
\Delta)}{\OO_X} = \OO_X$ when   $(X, \Delta)$ has log-terminal
singularities. \end{proof}
\begin{remark} \label{Lipman.Teissier.BS}
The inclusion $\frc^m \subseteq (f_1, \ldots, f_r)^{m+1-r}$ in the
last statement of the corollary holds more generally on any variety
with only rational singularities thanks to Lipman and Tessier's form
of  the Briancon-Skoda theorem  \cite[Thm 2.1]{LT}.

\end{remark}

\section{Proof of Theorem A}

We now refine the arguments of \cite{LL} to prove Theorem A.

  As in the statement, let $(\OO, \frakm)$ be the local ring of $X$ at the Cohen-Macaulay point $x \in X$. Let \[ \frc \ =\  (z_1, \ldots, z_d) \subseteq \OO\]  denote the ideal generated by the given system of parameters, and write
  \[ \JJ(\frc^m\cdot \frb^\lambda)  \, = \,\MMIP{(X,\Delta)}{\frc^m \cdot\frb^\lambda}_x \ \subseteq \ \OO \] for the germ at $x$ of the indicated multiplier ideal.

  We claim to begin with  that the map
  \begin{equation} \Tor_1\big( \frc^{d-1} \cdot \JJ , \OO / \frc \big) \lra    \Tor_1\big(  \JJ , \OO / \frc \big)\tag{*} \end{equation}
  vanishes. This follows by observing that \cite[Theorem B]{LL} remains valid in our setting, but it is more instructive to write out the argument explicitly. In fact, since $(\OO, \frakm)$ is Cohen-Macaulay and $\frc$ is generated by a regular sequence, we may compute the Tor's in question via the Koszul complex $K\bull(z_1, \ldots, z_d)$ associated to $z_1, \ldots, z_d$. This being said, consider the commutative diagram:
\begin{equation*}
\xymatrix@C-0pt{
  \OO^{\binom{d}{2}} \otimes \frc^{d-1} \MI{\frb^\lambda}
\ar[r] \ar@{^{(}->}[d] & \OO^{d} \otimes \frc^{d-1}
\MI{\frb^\lambda} \ar[r]\ar@{^{(}->}[d]  &  \OO \otimes \frc^{d}
\MI{\frb^\lambda}  \ar@{^{(}->}[d]
\\
  \OO^{\binom{d}{2}} \otimes \MI{\frc^{d-2} \frb^\lambda}
\ar[r]\ar@{^{(}->}[d]  & \OO^{d} \otimes \MI{\frc^{d-1}
\frb^\lambda} \ar[r] \ar@{^{(}->}[d]  &\OO  \otimes \MI{\frc^{d}
\frb^\lambda}
  \ar@{^{(}->}[d]
  \\   \OO^{\binom{d}{2}} \otimes  \MI{\frb^\lambda}
\ar[r] & \OO^{d} \otimes  \MI{\frb^\lambda} \ar[r] & \OO  \otimes
\MI{\frb^\lambda}.
  }
\end{equation*}
   Here the top and bottom rows arise from the Koszul complex (except that we have harmlessly modified the upper term on the right), and the middle row is part of the Skoda complex. The inclusion of the top into the middle row comes from Lemma \ref{module.property}.

   The groups in (*) are computed respectively as the homology of the first and third rows in the diagram, with the map arising from the inclusion of the one in the other. On the other hand, the middle row of the diagram is exact thanks to Theorem \ref{Skoda.Cxs.Exact}. Hence the map in (*) is zero, as required.

Folllowing the idea of \cite[Proposition 2.1]{LL}, we now deduce
Theorem A from (*).  Let  $ \JJ = \MI{\frb^\lambda}  \subseteq \OO$,
and consider a minimal free resolution $F\bull$ of $\MIJ$:
\begin{equation} \label{resoln.eqn}
\xymatrix{
 \ldots \ar[r]^{u_3} & F_2 \ar[r]^{u_2} &   F_1 \ar[r]^{u_1} & F_0
\ar[r]^{\pi}   & \MIJ \ar[r] & 0
  },
\end{equation}
where $F_i = \OO^{b_i}$. Assume for a contradiction that the
statement of the Theorem fails. Then there is a minimal generator $e
\in F_1$ such that $ u_1(e) \in (z_1, \dots, z_d)^{d} F_{0}$. In
particular, $e$ lies in the kernel of the induced map
$$
\xymatrix{
  F_1  \otimes \OO/\frc \ar[r]^{u_1 \otimes 1} &   F_{0}  \otimes \OO/\frc
    },
$$
and so represents a  class \[  \overline e \ \in \
 \Tor_1(\MIJ,  \OO/ \frc \big) \, =
 \, H_1 \big( F\bull \otimes  \OO/ \frc \big). \]
 Furthermore, since $e$ is {\it minimal} generator of $F_1$,  one has $e \not \in \frakm F_1$ and hence $e \not \in \text{im}(  u_2)$;  this ensures that the class $\overline e$ it represents in Tor is non-zero.  To complete the proof of Theorem A, we will show that $\overline{e}$ lies in the image
 of the natural  map
 \begin{equation}
\overline{e} \ \in \ \text{im} \Big( \Tor_1(\frc^{d-1}\MIJ,
\OO/\frc) \lra \Tor_1(\MIJ, \OO/\frc) \Big), \tag{**}
\end{equation}
which contradicts (*).

For (**), the plan is to explicate the representation of $e$ as a
Koszul cohomology class. To this end, let $h_1, \ldots, h_r$ be
minimal generators of $\JJ$, and let $g_1, \ldots, g_r \in \frakm$
be the coefficients of the minimal syzygy represented by $e \in
F_1$, so that $\sum g_ih_i = 0$. By assumption,
\[ g_i \ \in \ \frc^d \, = \, (z_1, \ldots, z_d)^d. \]
Now write
$$
g_i = z_1 g_{i1} + \dots +  z_d g_{id}
$$
where each $g_{ij} \in (z_1, \dots, z_d)^{d-1},$ and for  $j = 1,
\dots, d$   put
\[ G_j \ = \  h_1g_{1j} + \dots + h_{r}g_{rj}. \]
Then $G_j \in (z_1, \dots, z_d)^{d-1}\MIJ$. Furthermore,
    the $G_{j}$ give a Koszul relation on the $z_j$, i.e.
$$
z_1 G_1 + \dots + z_d G_d = 0,
$$
and so they represent a first cohomology class of the complex
 $(z_1, \dots,z_d)^{d-1}\JJ \otimes K_{\bullet}(z_1, \dots, z_d)$, where  as above $K_{\bullet}(z_1, \dots, z_d)$ denotes the Koszul complex on the $z_j$. In other words, $(G_1, \dots, G_d)$ represents an element
 \[ \eta \ \in  \ \Tor_1(\frc^{d-1}\MIJ,
{\OO}/{\frc}).\] It is not hard to check that the image of $\eta$
under the natural map to $\Tor_1(\MIJ, {\OO/}{\frc})$ is precisely
the  class $\overline e$. In other words, $\overline e$ lies in the
image of the map in (**), as required.

\begin{remark}\label{rem}
The "lifting" argument of \cite[Prop 1.1]{LL} can not be carried out
in the singular case for $p^{\text{th}}$ syzygies when $p \geq 2$
because the entries of the matrices defining the maps in the minimal
free resolution of $J$ can only be assumed in the maximal ideal
$\frakm$, not in $(z_1, \dots, z_d)$. However, if we happen to know
that an ideal $J$ has a minimal free resolution in which the entries
of the matrices describing all the maps $\mu_i$ for $i < p$ lie in
the ideal $(z_1, \dots, z_d)$, then we can carry out the same
"zig-zag" argument as in \cite[Prop 1.1]{LL} to deduce that no
minimal $p$-th syzygy of $J$ is in $(z_1, \dots, z_d)^{d+1 - p}$.
This will be the case, for example, for ideals $J$ that are
generated by a regular sequence of elements vanishing to high order
at $x$.
\end{remark}

\section{Corollaries and Examples}

We start with an example to show that the results of \cite{LL} do
not extend without change to the singular case.

\begin{example}\label{example}
Let $\OO$ be the local ring at the origin  of the hypersurface  in
$\CC^3$ defined by the equation \[ x^n + y^n + z^n  \ = \ 0 \] where
$n \ge3 $.  By blowing up the singular point, we get a log
resolution and it is easy to compute that the multiplier ideal of
the trivial ideal is precisely $\tau = (x, y, z)^{n-2}$.  This ideal
has a minimal syzygy vanishing to  order two:  the elements  $x^2,
y^2, $ and $z^2$ give a minimal syzygy on the (subset of the)
minimal generators $x^{n-2}, y^{n-2}, z^{n-2}$ of $\tau$. This shows
that in the singular case, Theorem A of \cite{LL}  does not hold as
stated, which would rule out a minimal syzygy vanishing modulo $(x,
y, z)^2$.  On the other hand, as our Theorem A predicts, this syzygy
does not vanish modulo the square of the ideal generated by two of
the coordinate functions. \qed
\end{example}

We next use results of Brian\c con-Skoda type to give statements
involving the multiplier ideal of $\OO_X$. Given a normal complex
variety $X$ of dimension $d$, define an ideal $\sigma(X) \subseteq
\OO_X$ by setting
\begin{equation}
\sigma(X) \ = \ \sum_{\Delta} \MMIP{(X, \Delta)} {\OO_X},
\end{equation}
the sum being taken over all effective $\QQ$-divisors $\Delta$ such
that $K_X + \Delta$ is $\QQ$-Cartier.

\begin{corollary} \label{Cor.With.Sigma}
If $x \in X$ is a Cohen-Macaulay point, and
\[ \JJ \, = \, \MMIP{(X, \Delta)}{\frb^\lambda}_x \ \subseteq \ \OO \,  =\,  \OO_{x,X} \]
 is the germ at $x$ of any multiplier ideal, then no minimal first syzygy of $\JJ$ can vanish modulo $\sigma(X) \cdot \frakm^{2d - 1}$.  I.e. if $g_1, \ldots, g_r$ are the coefficients of a minimal syzygy on minimal generators $h_1, \ldots, h_r \in \JJ$, then
 \[  g_i \ \not \in \ \sigma(X) \cdot \frakm^{2d-1}
 \] for at least one index $i$.
\end{corollary}

\begin{corollary} \label{Admissible.SIng.Cor}
If $X$ supports a $\QQ$-divisor $\Delta_0$ such that $(X, \Delta_0)$
is KLT, then no first syzygy of any multiplier ideal $\JJ$ can
vanish modulo $\frakm^{2d-1}$. $\qed$
\end{corollary}

\begin{proof}
[Proof of Corollary \ref{Cor.With.Sigma}] Let $z_1, \ldots, , z_d$
be a system of parameters at $x$ generating a reduction of the
maximal ideal $\frakm \subseteq \OO$. It follows from  Corollary
\ref{Cor.BS.Type} (ii) that
\[  \sigma(X) \cdot \frakm^{2d-1} \ \subseteq \ (z_1, \ldots, z_d)^d. \]
The assertion then follows from Theorem A.
\end{proof}

\begin{remark} Using the result of Lipman-Teissier \cite{LT} quoted in Remark \ref{Lipman.Teissier.BS}, a similar argument shows that the conclusion of Corollary \ref{Admissible.SIng.Cor} holds at any Cohen-Macaulay point of a $\QQ$-Gorenstein variety with only rational singularities. \qed
\end{remark}

\begin{example}\label{ex2}
Let $R$ be the local ring at the vertex of the affine cone over  a
smooth projective hypersurface of degree $n$ in projective $d$
space. Then $R$ is a $d$-dimensional Gorenstein ring
 with multiplier ideal $\tau = \frakm^{n-d}$. According to Corollary  \ref{Cor.With.Sigma}, no minimal syzygy of any multiplier ideal can vanish to order $n-d + (2d -1) = n + d -1 $.
   Note that since every ideal is contained in the unit  ideal, every multiplier ideal is contained in the multiplier ideal $\tau = \frakm^{n-d}$ of the trivial ideal.
 \end{example}

\begin{remark}
If $X$ is $\QQ$-Gorenstein, then $\sigma(X) = \MI{X,\OO_X}$, since
in this case one can take $\Delta = 0$ in the sum defining
$\sigma(X)$. It is known in this setting that $\MI{X,\OO_X}$ reduces
modulo $p \gg 0$ to the test ideal $\tau(X)$ of $X$ defined using
tight closure (see \cite{Ha2}, \cite{Sm2}).  It would be interesting
to know whether there is an analogous interpretation of the ideal
$\sigma(X)$ on an arbitrary normal variety $X$. In this connection,
observe  from Corollary \ref{Cor.BS.Type}  that if $f_1, \ldots, f_r
\in \frakm$ are functions generating a reduction of an ideal $\frc$,
then
\[  \sigma(X) \cdot \frc^m \ \subseteq \ (f_1, \ldots, f_r)^{m+1 -r}; \]
in characteristic $p > 0$ the analogous formula holds with
$\sigma(X)$ replaced by $\tau(X)$.  \qed
\end{remark}

Our remaining applications make more systematic use of the
connection with tight closure alluded to above.  Let $X$ be a
$\QQ$-Gorenstein variety of dimension $d$, $x \in X$ a
Cohen-Macaulay point, and set
\[  \tau \ = \ \MI{X, \OO_X}_x \ \subseteq \ \OO = \OO_{x,X}. \]
In the $\QQ$-Gorenstein setting, Corollary \ref{Cor.With.Sigma}
 asserts that no minimal first syzygy of a multiplier ideal can vanish modulo $\tau \cdot \frakm^{2d-1}$.
On the other hand, according to \cite[Thm 3.1]{Sm2} or \cite{Ha2},
the ideal $\tau$  is  a {\it universal test ideal}  for $\OO$ in the
sense of tight closure. Roughly speaking, this means that, after
reducing modulo $p$ for $p \gg 0$, the ideal  $\tau$ becomes the
test ideal for the corresponding ring $\OO$ modulo $p$, which is to
say, the elements of $\tau$  multiply  the tight closure of any
ideal $I$ back into the ideal $I$. For precise statements we refer
to the main theorems of either  \cite{Sm2} or \cite{Ha2}.
 Using this, we can deduce some statements in characteristic zero by reducing mod $p$ and invoking facts from tight closure.

 For example:

\begin{corollary}\label{cor3}
Assuming that $X$ is $\QQ$-Gorenstein, let $J \subset \OO$ denote
the Jacobian ideal of $\OO$ with respect to some local embedding in
a smooth variety. If $f = (f_1, \dots, f_r) $ is a minimal (first)
syzygy of some multiplier ideal  $\MIJ$, then  some
 $f_i$ fails to be in the ideal  $J \frakm^{2d-1}$.
\end{corollary}

\begin{proof}
Keeping in mind what was said in the preceding paragraph, it
suffices to show that the Jabobian ideal $J$ is contained in the
multiplier ideal $\tau$. Indeed, by \cite[Thm 3.4]{HH2}, in prime
characteristic, the Jacobian ideal is contained in the test ideal,
which means that in characteristic zero, the Jacobian ideal must be
contained in the multiplier ideal of the unit ideal by \cite[Thm
3.1]{Sm2}.\footnote{Related results comparing multiplier ideals and
Jacobian ideals can be found in \cite[Sec 4]{ELVS}.}
\end{proof}

\begin{remark}\label{rem2}
One can replace the multiplier ideal  $\tau=\MI{X, \OO_X} $ in this
discussion by any ideal $\tau\pr$ with the property that, after
reducing modulo $p$ for $p \gg 0$,  $\tau ^\pr$ is contained in the
{\it parameter test ideal\/} for the corresponding prime
characteristic ring.  The point is that the parameter test ideal
will multiply the tight closure of any ideal $I$ generated by
monomials in a system of parameters back into $I$, so that the
equation
\[
\tau^\pr \frakm^{2d-1} \ \subseteq \ (z_1, \ldots, z_d)^d
\] will hold for such  $\tau^\pr$.
 For example, one could replace $\sigma(X)$ in the statement of Corollary \ref{Cor.With.Sigma}
by a {\it universal parameter test ideal} if one is known to exist.
For example, in  the case where $\OO$ is rationally singular, the
universal parameter test ideal exists and is the unit ideal; this is
essentially the well-known statement that rationally singular rings
correspond, after reduction modulo $p$ for $p \gg 0$, to rings in
which all parameter ideals are tightly closed (see the main theorems
in \cite{Sm1}, and \cite{Ha1} or \cite{MS}.)
\end{remark}

 \begin{example}\label{ex3}
 Using Remark \ref{rem2}, we can generalize Example \ref{ex2} as follows.
 Let $x$ be the vertex of the cone over any rationally singular projective variety $Y$ with respect to any ample invertible sheaf $\LL$. In other words,  the local ring  $\OO$  at $x$ is obtained by localizing the section ring of $Y$ with respect to $\LL$ at its unique homogeneous maximal ideal $\frakm$. Assume that $\OO$ is Cohen-Macaulay and $\QQ$-Gorenstein (it is always normal), and let $d$ be its dimension. Let $a$ be the {\it a-invariant} of $\OO$, which is to say,  let $a$ be  the largest integer $n$ such that
the graded module $\omega_{\OO}$ is non-zero in degree $-n$ (or
alternatively, such that $\omega_X \otimes \LL^{-n}$ has a non-zero
global section).  Then no minimal first syzygy of  $\MIJ$ can vanish
to order $a + 2d$ at the vertex of the cone.

The point is in this case that, after reducing modulo $p \gg 0$, the
parameter test ideal includes all elements of degree greater than
$a$. In particular, $\frakm^{a+1}$ is contained in the parameter
test ideal, and so we can apply Remark \ref{rem2}.

To see that every element of degree greater than the $a$-invariant
is  contained in  the parameter test ideal, recall first that the
parameter test ideal of $\OO$ is the annihilator of the tight
closure of the zero module in
 $H^d_{\{x\}}(\OO)$ \cite[Prop 4.4]{Sm3}.
 On the other hand, for a section ring over a rationally singular variety,
  the tight closure of zero in  $H^d_{\{x\}}(\OO)$ is precisely the submodule of  $H^d_{\{x\}}(\OO)$ of non-negatively graded elements. Since  $H^d_{\{x\}}(\OO)$  vanishes in degree greater than $a$, it follows that every element of degree greater than $a$ annihilates the required tight closure module; see also \cite{HS}. \qed
 \end{example}

 \begin{example} The case of a standard graded algebra gives a user friendly special case of Example \ref{ex3}. Let $R$ be a normal Cohen-Macaulay $\QQ$-Gorenstein $\NN$-graded domain, generated by its degree one elements over its degree zero part $\CC$. Assume also that $R$ has isolated
 non-rational singularities. Then the first minimal syzygies of every multiplier ideal in $R$ have degree less than  $2d + a$, where $d$ is the dimension of $R$ and $a$ is the $a$-invariant of $R$.
(The statement holds even in the non-homogeneous case, where by
degree we mean the degree of the smallest degree component of the
syzygy.)
\end{example}

\begin{example} \label{ex4} In the smooth two-dimensional case, every integrally closed ideal is a multiplier ideal by a theorem of \cite{LW} or \cite{FJ}. As in \cite{LL}, Theorem A easily implies the existence of integrally closed ideals in dimensions $\ge 3$ that are not multiplier ideals also in the singular case.
For example, if $X$  is a normal Cohen-Macaulay $\QQ$-Gorenstein
variety of dimension  at least three, then we can find plenty of
regular sequences $f_1, \dots, f_r$, where $r$ is strictly smaller
than the dimension,  contained in an arbitrarily  high power  some
minimal reduction of $(z_1, \dots, z_d)$, of $\frakm$. For general
such $f_i$, the ideal $I$ they generate is radical and hence
integrally closed. On the other hand, the Koszul syzygies on the
$f_i$ violate Theorem A. If we wish to get an $\frakm$-primary
example, we can work in the graded case and add a large power of
$\frakm$ to $I$ as in \cite[Lemma 2.1]{LL}. (The proof of this lemma
is given in \cite{LL} for  polynomial rings, but it works for any
graded ring.)
\end{example}

Finally, it would be interesting to know the answer to the following
\begin{question}
If $R$ is a two dimensional $\QQ$-Gorenstein rationally singular
ring essentially of finite type over $\CC$, is every integrally
closed ideal a multiplier ideal?
\end{question}
\noi The first point to consider would be whether the conclusion of
Corollary \ref{Admissible.SIng.Cor}
 automatically holds in such rings.

\end{document}